\documentclass[12pt]{amsart}

\usepackage{url,graphicx,tabularx,array,geometry}
\usepackage{graphicx}
\usepackage{amssymb}
\usepackage{fancyhdr}
\usepackage{wrapfig}
\usepackage{epstopdf}
\usepackage{amsmath}
\usepackage[labelfont=bf]{caption}
\usepackage[hidelinks]{hyperref}
\usepackage{amsthm}
\usepackage{bbm}
\usepackage[T1]{fontenc}
\usepackage{yfonts}
\usepackage{harpoon}
\usepackage[all]{xy}
\usepackage{tikz}
\usepackage{hyperref}
\hypersetup{
 colorlinks,
 linkcolor={red!50!black},
 citecolor={blue!50!black},
 urlcolor={blue!80!black}
}

\usetikzlibrary{decorations.markings}
\tikzstyle{vertex}=[circle, draw, inner sep=0pt, minimum size=6pt]

\makeatletter
\def\imod#1{\allowbreak\mkern10mu({\operator@font mod}\,\,#1)}
\makeatother

\theoremstyle{plain}

\newtheorem{theoremN}{Theorem}[section]

\newtheorem{corollaryN}[theoremN]{Corollary}
\newtheorem{lemmaN}[theoremN]{Lemma}

\theoremstyle{definition}

\theoremstyle{remark}

\newtheorem*{note}{Note}

\newcommand{\tightoverset}[2]{%
  \mathop{#2}\limits^{\vbox to -.5ex{\kern-0.75ex\hbox{$#1$}\vss}}}

\usepackage[maxbibnames = 50]{biblatex}
\renewbibmacro{in:}{}
\DeclareFieldFormat{pages}{#1}
\begin{document}
\title[Structural Equivalence in Graphs and Complete Skeletons]{Structural Equivalence in Graphs and Complete Skeletons}

\author[Higgins]{Jonathan Higgins}

\address{Jonathan Higgins, Mathematics and Computer Science Department, Wheaton College, 501 College Ave, Wheaton, IL 60187, USA}
\email{jonathan.higgins@my.wheaton.edu}

\subjclass[2010] {Primary: 05C25, Secondary: 05C50}
\keywords {graphs, automorphisms, eigenvalues}
\maketitle

\begin{abstract}
     Two vertices $u$ and $v$ of a graph $\Gamma$ are strucuturally equivalent if and only if the transposition $(u\,v)$ is in Aut($\Gamma$), the automorphism group of $\Gamma$. Some properties of structural equivalence and the group of vertex permutations generated by the transpositions in Aut($\Gamma$) are discussed, along with the prime graphs of these groups. The notion of structural equivalence is used to develop a way of reconfiguring graphs into what are called their complete skeletons, which is closely related to compression graphs. Finally, the complete skeleton of a graph $\Gamma$, denoted $\Omega(\Gamma)$, is used to find a formula for rank$(I+A(\Gamma))$, which is helpful for determining the multiplicity of the -1 eigenvalue of $\Gamma$.
\end{abstract}

\maketitle

\section{Introduction}

The notion of vertex similarity in graphs is well understood. Two vertices $u$ and $v$ in a graph $\Gamma$ are said to be similar if there is some automorphism $\rho \in \textrm{Aut}(\Gamma)$ such that $\rho(u)=v$. This idea can be found in just about any graph theory textbook (see, for example, \cite{GT_book}). In this paper, we will consider a stronger notion of similarity of graph vertices known as structural equivalence. In particular, we can say that two vertices $u$ and $v$ in $\Gamma$ are structurally equivalent if their transposition is in Aut$(\Gamma)$. In section 2, this idea will be considered in depth and it is used to construct a group for studying the structure of graphs, which we will call the structural equivalence permutation (SEP) group. This information is known in the literature, so its purpose is to familiarize the reader with this topic rather than establishing novel results.

\bigskip
In the third section, we study the prime graphs (or Gruenberg-Kegel graphs) of SEP groups of graphs. The prime graph $\Gamma_G$ of a group $G$ is obtained by taking the set of prime divisors of the group order $|G|$ as its vertex set and, if $p,q \in V(\Gamma_G)$, $pq \in E(\Gamma_G)$ if and only if there is an element in $G$ of order $pq$. Prime graphs of groups have been studied in many other contexts and a number of characterizations have already been proved. For example, it was proved in \cite{2015_REU_Paper} that a graph $\Gamma$ is isomorphic to the prime graph of a solvable group if and only if $\overline{\Gamma}$ is triangle-free and 3-colorable. In this paper, we are not concerned with proving new characterizations, but we are interested in seeing how they can inform us about the structure of SEP groups. Additionally, we will consider chains of prime graphs of SEP groups, proving that they must all terminate.

\bigskip
In the final two sections of the paper, we look at compressions of graphs and develop the idea of using complete skeletons for effectively conveying information about graphs after being compressed. Finally, we show that the problem of determining the multiplicity of the -1 eigenvalue for a graph $\Gamma$ can be reduced to a problem regarding the complete skeleton of $\Gamma$, denoted $\Omega(\Gamma)$.

\section{Structurally Equivalent Vertices and the SEP Group}

Two vertices $u,v$ of a graph are said to be similar if there is an automorphism $\rho$, which maps $u$ to $v$. The notion of vertex-transitive graphs arises when all vertices of a graph are similar. Here, though, we will consider a stronger, but related, condition on vertices in a graph. 

\bigskip
\noindent
\textbf{Definition.} Two vertices $u, v \in V(\Gamma)$ are \textbf{structurally equivalent} if and only if $(u\, v) \in$ Aut$(\Gamma)$. In other words, the permutation switching just the two vertices must be an automorphism.

\bigskip
From this definition, we see that two structurally equivalent vertices are essentially the same in an unlabelled graph, and if we swap two structurally equivalent vertices in a labelled graph, this does not affect the edge set. In the theorem below, we will use the notation $N^i(x)$ to denote the $i$th neighborhood of a vertex $x$ in a graph $\Gamma$, meaning $N^i(x) =\{y \in V(\Gamma) : d(x,y)=i\}$.

\begin{theoremN}
\label{neighborhoods}
In a simple connected graph $\Gamma$ with at least two vertices, two connected vertices $u,v$ are structurally equivalent if and only if $N^i(u) \setminus \{v\} = N^i(v) \setminus \{u\}$ for all $i \leq \textrm{diam}(\Gamma)$.
\end{theoremN}

\begin{proof}
$(\Longrightarrow)$ Suppose $(u\, v) \in \textrm{Aut}(\Gamma)$. Clearly, $N^1(u) \setminus \{v\}=N^1(v) \setminus \{u\}$ because the transposition $(u \, v)$ must preserve adjacency relations. Next, assume there is some $i \leq \textrm{diam}(\Gamma)$ such that $N^i(u) \setminus \{v\} \neq N^i(v) \setminus \{u\}$. Thus, we can say without loss of generality that there is some $x \in V(\Gamma)$ such that $d(u,x) > d(v,x)$. This cannot be, though, because $N^1(u)\setminus \{v\}=N^1(v)\setminus\{u\}$, which is non-empty if $|V(\Gamma)|>2$. This means that a shortest path between $u$ and $x$ can also be a shortest path between $v$ and $x$, simply by replacing the $u$ with the $v$. Hence, $d(u,x)=d(v,x)$ for all $x \in V(\Gamma) \setminus \{u,v\}$, so $N^i(u)\setminus \{v\}=N^i(v)\setminus\{u\}$. If $|V(\Gamma)|=2$, then the statement holds trivially.
\\
\\
$(\Longleftarrow)$ Suppose $N^i(u) \setminus \{v\}=N^i(v) \setminus \{u\}$ for all $i \leq \textrm{diam}(\Gamma)$. This means $u$ and $v$ can be swapped without affecting the adjacency relations of $\Gamma$. Therefore, $(u \, v) \in \textrm{Aut}(\Gamma)$, which tells us that $u$ and $v$ are structurally equivalent. 
\end{proof}

\bigskip
Note that it is important to subtract $v$ from $N^i(u)$ and $u$ from $N^i(v)$ because our graph is simple. However, this is only relevant when $i=1$ or 2 because, for any two structurally equivalent vertices $u$ and $v$, $d(u,v) \leq 2$. Hence, we need $N^1(u) \setminus \{v\}=N^1(v) \setminus \{u\}$ when $u$ and $v$ are adjacent and $N^2(u) \setminus \{v\}=N^2(v) \setminus \{u\}$ when $u$ and $v$ are not adjacent. We will keep the notation $N^i(u) \setminus \{v\}=N^i(v) \setminus \{u\}$, though, so that we can cover both cases at once.

\bigskip
It is not difficult to see that structural equivalence forms an equivalence relation on the vertex set of a graph, which we see in the following theorem. In other words, the vertices of a graph can be partitioned into sets of structurally equivalent vertices. The proof is omitted because this is a well-known result in graph and network theory. Algorithms have been developed to determine these equivalence classes, which can be seen in \cite{algorithm}

\begin{theoremN}
Structural equivalence forms an equivalance relation on the vertex set of a graph. 
\end{theoremN}


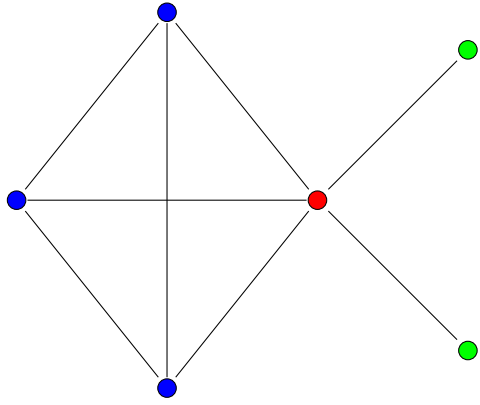
\begin{figure}
    \centering
    \begin{tikzpicture}
     \node (n1) at (1,5) {};
  \node (n2) at (3,7.5)  {};
  \node (n3) at (3,2.5)  {};
  \node (n4) at (5,5) {};
  \node (n5) at (7,7)  {};
  \node (n6) at (7,3)  {};
  
  \draw [fill=blue] (1,5) circle (3.5pt);
\draw [fill=blue] (3,7.5) circle (3.5pt);
\draw [fill=blue] (3,2.5) circle (3.5pt);
\draw [fill=red] (5,5) circle (3.5pt);
\draw [fill=green] (7,7) circle (3.5pt);
\draw [fill=green] (7,3) circle (3.5pt);

 \foreach \from/\to in {n1/n2,n1/n3,n1/n4,n2/n3,n2/n4,n3/n4,n4/n5,n4/n6}
    \draw (\from) -- (\to);
    \end{tikzpicture}
    \caption{The blue, red, and green vertices represent the the three sets of structurally equivalent vertices that partition the vertex set of the pineapple graph $K_4^2$.}
    \label{pineapple graph}
\end{figure}

\bigskip
\noindent
\textbf{Definition} The \textbf{structural equivalence classes} of a graph $\Gamma$ are the sets of structurally equivalent vertices that partition $V(\Gamma)$.

\bigskip
For the remainder of the paper, we will simply use "equivalence classes" to refer to the structural equivalence classes in a graph.

\bigskip
Just as we can discuss the automorphism group of a graph, so can we formulate a group of permutations between structurally equivalent vertices in a graph, which we define below.

\bigskip
\noindent
\textbf{Definition.} For a graph $\Gamma$, we call the group of permutations between structurally equivalent vertices in $\Gamma$ the \textbf{structural equivalence permutation group} of $\Gamma$, which we write as $SEP(\Gamma)$. 

\bigskip
One way that we can express this group is 
\[
SEP(\Gamma) = \langle \{(u\,v): (u\,v) \in \textrm{Aut}(\Gamma)\} \rangle.
\] 

Because structural equivalence is an equivalence relation, it is easy to see that the SEP group of an equivalence class is just the symmetric group $S_n$ when there are $n$ vertices in the equivalence class. Hence, because $K_n$ has just one equivalence class, we see $SEP(K_n) = \textrm{Aut}(K_n) = S_n$. However, for any graph that is not complete, there are at least two equivalence classes, so $SEP(\Gamma)$ has a minimum of two symmetric permutation subgroups on disjoint sets of vertices. We will let $s(\Gamma)$ denote the number of equivalence classes for $\Gamma$. Next, we let the permutation group on the $i$th equivalence class $(\textrm{where}\, 1 \leq i \leq s(\Gamma))$ be written as $S_{\gamma_i,i}$, where $\gamma_i$ is the number of vertices in this $i$th equivalence class. Also, note that $S_{\gamma_i,i} \cap S_{\gamma_j,j} = 1$ for $i \neq j$ because the symmetric groups are on disjoint sets of vertices. This group is valuable for the studying of chemical structures, as seen in \cite{chemistry}. Also, algorithms have been designed to determine these groups for particular graphs (see \cite{chemistry_algorithm2} and \cite{chemistry_algorithm}). 

\begin{theoremN}
\label{SP theorem}
We can write the SEP group of $\Gamma$ as \begin{equation}
    SEP(\Gamma) = \big{\langle} \bigcup S_{\gamma_i,i} : 1\leq i \leq s(\Gamma) \big{\rangle} \cong \prod_{i=1}^{s(\Gamma)} S_{\gamma_i,i}.
\end{equation}
Also, $|SEP(\Gamma)| = \prod_{i=1}^{s(\Gamma)} \gamma_i !$.
\end{theoremN}
\begin{proof}
Clearly, $SEP(\Gamma) \subseteq \big{\langle} \bigcup S_{\gamma_i,i} : 1\leq i \leq s(\Gamma) \big{\rangle}$ because there is some $1 \leq i \leq s(\Gamma)$ such that $(u\, v) \in S_{\gamma_i,i}$ for all $(u\,v) \in \mathrm{Aut}(\Gamma)$, which tells us that all the generators of $SEP(\Gamma)$ are included in $\big{\langle} \bigcup S_{\gamma_i,i} : 1\leq i \leq s(\Gamma) \big{\rangle}$. Next, consider some $\sigma \in \big{\langle} \bigcup S_{\gamma_i,i} : 1\leq i \leq s(\Gamma) \big{\rangle}$. Because each of the $S_{\gamma_i,i}$ groups consist of the permutations on distinct vertex sets, the smallest subgroup containing $\bigcup S_{\gamma_i,i}$ for all $1 \leq i \leq s(\Gamma)$ is the group consisting of all possible compositions of distinct elements from each of the $S_{\gamma_i,i}$. Because $S_{\gamma_i,i} \cap S_{\gamma_j,j} = 1$ for $i \neq j$, we see that $\big{\langle} \bigcup S_{\gamma_i,i} : 1\leq i \leq s(\Gamma) \big{\rangle}$ can be expressed as the direct product of the $S_{\gamma_i,i}$, i.e. $\big{\langle} \bigcup S_{\gamma_i,i} : 1\leq i \leq s(\Gamma) \big{\rangle} \cong \prod_{i=1}^{s(\Gamma)} S_{\gamma_i,i}$. We find $\big{\langle} \bigcup S_{\gamma_i,i} : 1\leq i \leq s(\Gamma) \big{\rangle} \subseteq SP(\Gamma)$ because each of the $S_{\gamma_i,i}$ can be generated by the transpositions of its vertices, and each element in $\big{\langle} \bigcup S_{\gamma_i,i} : 1\leq i \leq s(\Gamma) \big{\rangle}$, as we saw, is a composition of distinct elements from the $S_{\gamma_i,i}$, so $\sigma \in \big{\langle} \bigcup S_{\gamma_i,i} : 1\leq i \leq s(\Gamma) \big{\rangle}$ can be generated by the transpositions in Aut$(\Gamma)$ via compositions. The formula for the order of $SEP(\Gamma)$ follows immediately from the fact that $SEP(\Gamma) \cong \prod_{i=1}^{s(\Gamma)} S_{\gamma_i,i}$.
\end{proof}

Just as the automorphism group of a graph conveys useful information about the structure of a graph, so does the SEP group. It is also not difficult to see that the SEP group is a subgroup of the automorphism group and it conveys a stronger notion of graph symmetry. Any two elements in a same permutation cycle in $SEP(\Gamma)$ can be transposed without affecting the adjacency relations of the graph. This is not necessarily the case for the automorphism group. This can be seen through the example in Figure \ref{sspg example fig}. If we denote this graph as $\Gamma$, then we find that $SEP(\Gamma)$ only consists of the identity permutation and $(3\,4)$. Consider the permutation $(1\,6)(2\,5)$. This is a valid automorphism, so it is in Aut$(\Gamma)$, but neither $(1\,6)$ nor $(2\,5)$ are in the automorphism group, which would be the case if $(1\,6)(2\,5) \in SEP(\Gamma)$. We see that this generalizes to the notion that $SEP(\Gamma)$ has the following hereditary property:

\begin{theoremN}
\textbf{(Hereditary Property of $SEP(\Gamma)$)} If $\sigma \in SEP(\Gamma)$, then any proper sub-cycle $\tau$ of a cycle in $\sigma$ is in $SEP(\Gamma)$. Also, if $\sigma$ consists of the compositions of disjoint cycles, then each of the disjoint cycles are in $SEP(\Gamma)$.
\end{theoremN}

In other words, if we have some cycle $(1\,2\,3) \in SEP(\Gamma)$, then we also know $(1\,2),(1\,3),$ and $(2\,3)$ are in $SEP(\Gamma)$. For the case where $\sigma$ is a transposition, any proper subcycle is the identity permutation, and if $\sigma$ consists of a collection of disjoint cycles, we can apply this to each of the cycles separately.

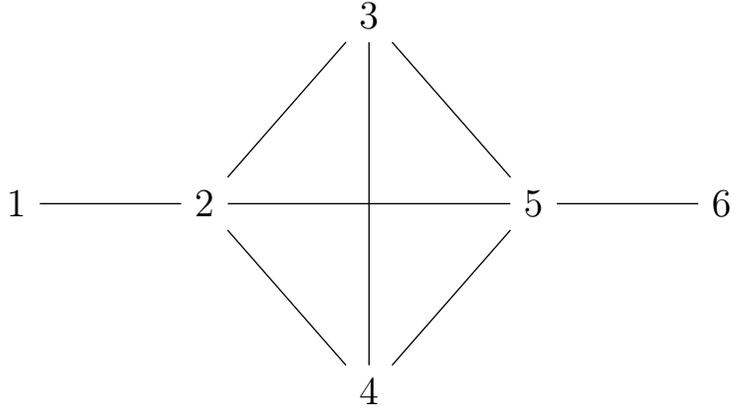
\begin{figure}
    \centering
    \scalebox{1.25}{
 \begin{tikzpicture}   
\node (n1) at (1,5) {1};
  \node (n2) at (3,5)  {2};
  \node (n3) at (4.75,7)  {3};
  \node (n4) at (4.75,3) {4};
  \node (n5) at (6.5,5)  {5};
  \node (n6) at (8.5,5) {6};

\foreach \from/\to in {n1/n2,n2/n3,n2/n4,n2/n5,n3/n4,n3/n5,n4/n5,n5/n6}
    \draw (\from) -- (\to);
   
   \end{tikzpicture}}
    \caption{An example demonstrating some differences between $SEP(\Gamma)$ and Aut$(\Gamma)$.}
    \label{sspg example fig}
\end{figure}

\begin{proof}
The theorem is clearly true for the case where the proper sub-cycles are transpositions, given that $SEP(\Gamma)$ is generated by a set of transpositions. Any other proper sub-cycle $\tau$ must also be an element of a unique $S_{\gamma_i,i}$, so $\tau \in SEP(\Gamma)$ by Theorem \ref{SP theorem}. Similar reasoning can be applied to the disjoint cycles composing $\sigma$.
\end{proof}

In the following section, we will consider some properties of prime graphs of the SEP groups of graphs.

\section{Prime Graphs of SEP Groups of Graphs}

Recall the definition of the prime graph of a group. For a group $G$, we define the prime graph of $G$ as the graph $\Gamma_G$, where $V(\Gamma_G)$ is the set of prime numbers dividing the order of $G$ and, for any two of these primes $p$ and $q$, $pq \in E(\Gamma_G)$ if and only if there is an element in $G$ of order $pq$. Next, we will need to define some new notation. We let $\tilde{\Gamma}^{(0)}$ be $\Gamma$ and $\tilde{\Gamma}^{(i)}$ be the prime graph of $SEP(\tilde{\Gamma}^{(i-1)})$. When we are only concerned with the first prime graph of $SEP(\Gamma)$ for some graph $\Gamma$, we will write $\tilde{\Gamma}$ instead of $\tilde{\Gamma}^{(1)}$. We will also adjust some of our notation for the permutation groups on the equivalence classes of a graph. The $S_{\gamma_i,i}$ such that $\gamma_i$ is the greatest shall be denoted as $S_{\alpha}$, and we will use $S_{\beta}$ to denote the $S_{\gamma_i,i}$ on the second largest equivalence class. Thus, for graphs $\Gamma$ with at least two equivalence classes, we can write 
\[
SEP(\Gamma) \cong S_{\alpha} \times S_{\beta} \times \prod_{i=1}^{s(\Gamma)-2} S_{\gamma_i,i}.
\]

\begin{theoremN}
\label{edges in tilde gamma graph}
If a graph $\Gamma$ has only one (structural) equivalence class on $\alpha$ vertices, then $V(\tilde{\Gamma})$ is the set of primes that are less than or equal to $\alpha$, and for any two of these primes $p,q \in V(\tilde{\Gamma})$, $pq \in E(\tilde{\Gamma})$ if and only if $p+q \leq \alpha$. If $\Gamma$ has at least two equivalence classes, such that $\alpha$ and $\beta$ are the sizes of the largest and second smallest equivalence classes, then $V(\tilde{\Gamma})$ is still the set of primes that are less than or equal to $\alpha$, but for two such primes, $p$ and $q$, $pq \in E(\tilde{\Gamma})$ if and only if $p+q \leq \alpha$ or, if $p>q$, $p \leq \alpha$ and $q \leq \beta$.  
\end{theoremN}

\begin{proof}
Suppose first that $\Gamma$ has only one equivalence class. Then $SEP(\Gamma) = S_{\alpha}$, where $\alpha$ is the size of this equivalence class, as expressed in the theorem. First, we observe 
\[
V(\tilde{\Gamma}) = \{p : p \,\textrm{is prime and} \,p \leq \alpha\}
\]
because $|SEP(\Gamma)| = \alpha!$, so the primes that divide $|SEP(\Gamma)|$ are precisely the primes less than or equal to $\alpha$. If $p,q \leq \alpha$ and are primes, then the only way to get an element of order $pq$ in $SEP(\Gamma)$ is if we can compose two disjoint cycles of sizes $p$ and $q$, which we can do if and only if $p+q\leq \alpha$. 
\\
\\
Next, suppose $\Gamma$ has at least two equivalence classes. Then, we can say
\[
SEP(\Gamma) \cong S_{\alpha} \times S_{\beta} \times \prod_{i=1}^{s(\Gamma)-2} S_{\gamma_i,i}.
\]
It is clear that reasoning from above also applies to this case, so all we have left to show is that for $p,q \in V(\tilde{\Gamma})$, even if $p+q > \alpha$, $pq$ can still be an edge in $\tilde{\Gamma}$, given that $p\leq \alpha$ and $q \leq \beta$ (where $p> q$). This follows from the fact that $S_{\alpha}$ and $S_{\beta}$ are permutation groups on different equivalence classes. If there is a cycle of size $p$ in $S_{\alpha}$ and a cycle of size $q$ in $S_{\beta}$, then the theorem follows.
\end{proof}

In the next theorem, we will characterize $\tilde{\Gamma}$ graphs by the presence $k$-cliques, which will also enable us to determine for which graphs the prime graph of $SEP(\Gamma)$ is complete. First, we will let $p_n$ denote the $n$th prime number.

\begin{theoremN}
\label{tilde gamma cliques}
If $SEP(\Gamma) \cong S_{\alpha} \times S_{\beta} \times \prod_{i=1}^{s(\Gamma)-2} S_{\gamma_i,i}$, then $\tilde{\Gamma}$ has a $k$-clique if and only if $\alpha \geq p_k+p_{k-1}$ or $\alpha \geq p_k$ and $\beta \geq p_{k-1}$. If $SEP(\Gamma) = S_{\alpha}$, then $\tilde{\Gamma}$ has a $k$-clique if and only if $\alpha \geq p_k+p_{k-1}$.
\end{theoremN}

By Theorem \ref{edges in tilde gamma graph}, we can deduce that if $\tilde{\Gamma}$ has a $k$-clique, then this clique is on the first $k$ primes in the labelled prime graph. Now, we will prove the theorem.

\begin{proof}
First, consider the case where $\Gamma$ has at least two equivalence classes, or $SEP(\Gamma) \cong S_{\alpha} \times S_{\beta} \times \prod_{i=1}^{s(\Gamma)-2} S_{\gamma_i,i}$.  By Theorem \ref{edges in tilde gamma graph}, we can deduce that $p_kp_{k-1}\in E(\tilde{\Gamma})$ implies that $p_ip_j \in E(\tilde{\Gamma})$ for all $p_i+p_j\leq p_k+p_{k-1}$. Hence, we just need to find the conditions under which $p_kp_{k-1}\in E(\tilde{\Gamma})$ is true in order to prove that $\tilde{\Gamma}$ has a $k$-clique. However, this is also taken care of by Theorem \ref{edges in tilde gamma graph}, so the first case is proved. The second case follows by similar reasoning.
\end{proof}

\begin{corollaryN}
\label{complete tilde gamma}
$\tilde{\Gamma} = K_n$ for $n>1$ if and only if $SEP(\Gamma) \cong S_{\alpha} \times S_{\beta} \times \prod_{i=1}^{s(\Gamma)-2} S_{\gamma_i,i}$ such that $p_{n+1}> \alpha \geq p_n$ and $\beta \geq p_{n-1}$.
\end{corollaryN}

\begin{proof}
Suppose $\tilde{\Gamma}$ is a complete graph on at least two vertices. Thus, it contains a $n$-clique and $\alpha<p_{n+1}$; otherwise, there would be at least $n+1$ vertices. By Theorem \ref{tilde gamma cliques}, $SEP(\Gamma) \cong S_{\alpha} \times S_{\beta} \times \prod_{i=1}^{s(\Gamma)-2} S_{\gamma_i,i}$ such that $p_{n+1} > \alpha \geq p_n+p_{n-1}$ or $p_{n+1}> \alpha \geq p_n$ and $\beta \geq p_{n-1}$, or $SEP(\Gamma)$ is just $S_\alpha$, where $p_{n+1}>\alpha\geq p_n+p_{n-1}$. However, it is known that there is no $n$ such that $p_{n+1}>p_n+p_{n-1}$ (This can be seen, for example, in \cite{mathoverflow_primes}). Thus, we see that the only option for $SEP(\Gamma)$ is $SEP(\Gamma) \cong S_{\alpha} \times S_{\beta} \times \prod_{i=1}^{s(\Gamma)-2} S_{\gamma_i,i}$ such that $p_{n+1}> \alpha \geq p_n$ and $\beta \geq p_{n-1}$. The reverse direction follows immediately from Theorem \ref{tilde gamma cliques} and restriction $p_{n+1}>\alpha$.
\end{proof}

\begin{theoremN}
\label{vertex set of repeated pg of sp decreases}
$\tilde{\Gamma}^{(i)} \subsetneq \tilde{\Gamma}^{(i-1)}$ for all $i$ such that $\tilde{\Gamma}^{(i-1)} \neq \emptyset$.
\end{theoremN}

\begin{proof}
First, we will show $V(\tilde{\Gamma}^{(i)}) \subsetneq V(\tilde{\Gamma}^{(i-1)})$. Suppose $|V(\tilde{\Gamma}^{(i-1)})|>1$. We note $SEP(\tilde{\Gamma}^{(i-1)})$ is isomorphic to the direct product of symmetric groups on the vertices in the equivalence classes of $\tilde{\Gamma}^{(i-1)}$. Because $|SEP(\tilde{\Gamma}^{(i-1)})| = \prod_{i=1}^{s(\tilde{\Gamma}^{(i-1)})} \alpha_i!$ by Theorem \ref{SP theorem}, we note 
\[V(\tilde{\Gamma}^{(i)}) \subseteq \{p : p\, \textrm{is prime and}\, p \leq |V(\tilde{\Gamma}^{(i-1)})|\} \subsetneq V(\tilde{\Gamma}^{(i-1)}).\]
(Note that the first set inclusion is an equality in the case that $\tilde{\Gamma}^{(i-1)}$ is complete.) Thus, we have found $V(\tilde{\Gamma}^{(i)}) \subsetneq V(\tilde{\Gamma}^{(i-1)})$. Note that in the special case where $|V(\tilde{\Gamma}^{(i-1)})|=1$, $\{p : p\, \textrm{is prime and}\, p \leq |V(\tilde{\Gamma}^{(i-1)})|\} = \emptyset$, so $V(\tilde{\Gamma}^{(i)})$ is also empty. If $\tilde{\Gamma}^{(i-1)} = \emptyset$, then it makes no sense to consider $\tilde{\Gamma}^{(i)}$ (although, we could say that, in a trivial sense, $\tilde{\Gamma}^{(i)} = \emptyset$, the empty graph).
\\
\\
All we have left to prove is that the edge set of $\tilde{\Gamma}^{(i)}$ is contained in the edge set of $\tilde{\Gamma}^{(i-1)}$. Consider the two graphs as labelled prime graphs and assume that there is some edge $pq \in E(\tilde{\Gamma}^{(i)})$ such that $pq \not\in E(\tilde{\Gamma}^{(i-1)})$. Because we have seen $V(\tilde{\Gamma}^{(i)}) \subsetneq V(\tilde{\Gamma}^{(i-1)})$, it follows from Theorem \ref{edges in tilde gamma graph} that the largest equivalence class in $\tilde{\Gamma}^{(i)}$ has fewer vertices than the largest equivalence class in $\tilde{\Gamma}^{(i-1)}$. Let $S_{\alpha^i}$ and $S_{\alpha^{(i-1)}}$ be the permutation groups on these largest equivalence classes for $\tilde{\Gamma}^{(i)}$ and $\tilde{\Gamma}^{(i-1)}$, respectively. Define $S_{\beta^i}$ and $S_{\beta^{i-1}}$ similarly for the permutation groups on the second largest equivalence classes, if they exist. Hence, we have found $\alpha^i < \alpha^{i-1}$. Also, by Theorem \ref{edges in tilde gamma graph}, we know that $V(\tilde{\Gamma}^{(i)}) = \{p: p \, \textrm{is prime and}\, p\leq \alpha^{(i-1)}$\}. Because equivalence classes partition the vertex set of a graph, we can say $\alpha^i + \beta^i \leq |\{p: p \, \textrm{is prime and}\, p\leq \alpha^{i-1}\}|$. Next, we can use Theorem \ref{edges in tilde gamma graph} again to determine that either $p+q \leq \alpha^i$ or $p\leq \alpha^i$ and $q \leq \beta^i$ (where $p> q$) because $pq \in E(\tilde{\Gamma}^{(i)})$. First, suppose $p+q \leq \alpha^i$. Then, we find $p+q < \alpha^{i-1}$ because $\alpha^i<\alpha^{i-1}$. This is a contradiction, though, because this implies $pq \in E(\tilde{\Gamma}^{(i-1)})$. Next, suppose $p\leq \alpha^i$ and $q \leq \beta^i$. We find $p+q \leq \alpha^i + \beta^i \leq |\{p: p \, \textrm{is prime and}\, p\leq \alpha^{i-1}\}| \leq \alpha^{i-1}$, so we have another contradiction. Thus, if $pq \in E(\tilde{\Gamma}^{(i)})$, then $pq \in E(\tilde{\Gamma}^{(i-1)})$, which completes the proof.
\end{proof}

With this observation, the next result follows immediately.

\begin{corollaryN}
There is no graph $\Gamma$ such that $\Gamma \cong \tilde{\Gamma}$. 
\end{corollaryN}

Embedded in this statement is the fact that there is no $i$ such that $\tilde{\Gamma}^{(i-1)} \cong \tilde{\Gamma}^{(i)}$ for some "initial graph" $\tilde{\Gamma}^{(0)} = \Gamma$. This is because we can just let $\Gamma$ be $\tilde{\Gamma}^{(i-1)}$.

\bigskip
\noindent
\textbf{Definition.}
Let $\Gamma = \tilde{\Gamma}^{(0)} \supsetneq \tilde{\Gamma}^{(1)} \supsetneq ... \supsetneq \tilde{\Gamma}^{(i)} \supsetneq ... \supsetneq \tilde{\Gamma}^{(n)} \supsetneq \emptyset$ be the \textbf{SEP $\Gamma$-series}. If $\tilde{\Gamma}^{(n)}$ is the final element in the series before the empty graph, we say that $\tilde{\Gamma}^{(n)}$ is the \textbf{minimal element} in the SEP $\Gamma$-series and $n$ is the \textbf{length} of the SEP $\Gamma$-series. 

\begin{theoremN}
\label{minimality in SP series}
Every SP $\Gamma$-series has a minimal element (i.e. has a finite length).
\end{theoremN}

\begin{proof}
By Theorem \ref{vertex set of repeated pg of sp decreases} and well-ordering, we know that there is an $i$ such that $\tilde{\Gamma}^{(i)}$ is minimal in the poset of the graphs $\tilde{\Gamma}^{(j)}$ with respect to $\supset$. The theorem follows.
\end{proof}

\section{The Connection Between Structural Equivalence and rank$(I+A(\Gamma))$}

Now, we will look at how structurally equivalent vertices affect the adjacency matrix of a graph. We will use the standard notation, $A(\Gamma)$, to denote the adjacency matrix of a graph $\Gamma$. The following theorem, regarding edges within equivalence classes, can be found in \cite{ucla}.


\begin{theoremN}
\label{big supersimilar theorem}
If a class of structurally equivalent vertices has an edge, then the induced subgraph on the equivalence class of vertices is complete. 
\end{theoremN}

\bigskip
Any graph $\Gamma$ can be expressed as a new graph where the vertices represent complete induced subgraphs of $\Gamma$ and if there is an edge between any two of these complete induced subgraphs $K_a$ and $K_b$, this indicates that all the vertices in $K_a$ are completely connected with the all the vertices in $K_b$. In other words, the edges denote the existence of induced complete bipartite graphs between two complete graphs (which is ultimately a larger complete graph). Note that for many graphs, if we try to reconfigure them in this way, we will find that every vertex is just $K_1$ and every edge only represents one edge in $\Gamma$, which means that this reconfiguration of $\Gamma$ is isomorphic to $\Gamma$. Such a reconfiguration of $\Gamma$ would, thus, be considered trivial. However, there is a considerable number of graphs such that this form of reconfiguration is non-trivial. Next, we will provide a definition and an example to clarify.

\bigskip
\noindent
\textbf{Definition.} The \textbf{complete skeleton} of a graph $\Gamma$, denoted $\Omega(\Gamma)$, is the smallest reconfiguration of $\Gamma$ such that the vertices of $\Omega(\Gamma)$ represent complete induced subgraphs of $\Gamma$ and the edges represent the complete connection of edges between the two complete induced subgraphs that it connects.

\begin{figure}
    \centering
    \begin{tikzpicture}
    \node (n33) at (11,8.2) {};
    \node (n34) at (11,5.3) {};
     \node (n35) at (11.5,6.75) {};
    \node (n36) at (12.625,7.8) {};
    \node (n37) at (12.625,5.7) {};
    \node (n38) at (13.75,6.75) {};
    \node (n39) at (14.25,8.2) {};
    \node (n40) at (14.25,5.3) {};
    
    \node (n1) at (16.75,6.75) {$K_3$};
    \node (n2) at (18,8.2) {$K_1$};
    \node (n3) at (18,5.3) {$K_1$};
    \node (n4) at (19.25, 6.75) {$K_3$};
    
    \draw [fill=black] (11,8.2) circle (3.5pt);
      \draw [fill=black] (11,5.3) circle (3.5pt);
      \draw [fill=black] (11.5,6.75) circle (3.5pt);
      \draw [fill=black] (12.625,7.8) circle (3.5pt);
       \draw [fill=black] (12.625,5.7) circle (3.5pt);
      \draw [fill=black] (13.75,6.75) circle (3.5pt);
      \draw [fill=black] (14.25,8.2) circle (3.5pt);
      \draw [fill=black] (14.25,5.3) circle (3.5pt);
      
       \foreach \from/\to in {n33/n34,n33/n35,n34/n35,n38/n39,n38/n40,n39/n40,n36/n33,n36/n34,n36/n35,n36/n38,n36/n39,n36/n40,n37/n33,n37/n34,n37/n35,n37/n38,n37/n39,n37/n40,n1/n2,n2/n4,n4/n3,n3/n1}
    \draw (\from) -- (\to);
    \end{tikzpicture}
    \caption{An example of a graph $\Gamma$ and its complete skeleton $\Omega(\Gamma)$.}
    \label{first complete skeleton ex.}
\end{figure}
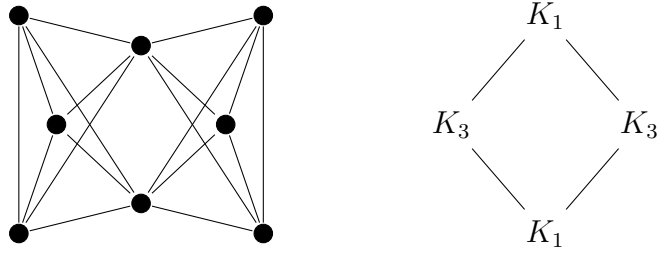
\bigskip
The complete skeleton of a graph $\Gamma$, as defined above, is very closely related with the compression graph of $\Gamma$. Compression graphs have been considered previously (see \cite{ucla}), but we will approach this problem in a slightly different way. 

\bigskip
Figure \ref{first complete skeleton ex.} provides an example to help visualize what complete skeletons of graphs may look like. The reader is encouraged to convince him- or herself that the example is, in fact, true before moving on. In some of the following theorems, we will draw the connection between complete skeletons and equivalence classes. Before doing so, though, it will be helpful to define one more term.

\bigskip
\noindent
\textbf{Definition.} For any reconfiguration of a graph $\Gamma$ via the reinterpretation of vertices and edges given above, which we will write as $R\Gamma$, we say that two vertices $K_a$ and $K_b$ in $R\Gamma$ can be \textbf{conflated} if and only if they are connected and $N^1(K_a)\setminus\{K_b\} = N^1(K_b)\setminus\{K_a\}$. The conflation of $K_a$ and $K_b$ results in a single vertex $K_{a+b}$, where all the vertices incident to $K_{a+b}$ are the edges that were incident to $K_a$ or $K_b$, excluding the edge between the two.

\bigskip
By this new definition, we see that $R\Gamma = \Omega(\Gamma)$ if and only if no vertices in $R\Gamma$ can be conflated.

\begin{theoremN}
\label{complete skeleton theorem}
The vertices of $\Omega(\Gamma)$ are the equivalence classes of $\Gamma$, with one possible exception. If a collection of independent vertices are $K_1$ and share the same neighbors, they form a completely disconnected equivalence class.
\end{theoremN}

\begin{proof}
First, we will consider the case where there are no $K_1$ vertices in $\Omega(\Gamma)$ that form a disconnected equivalence class. It is apparent that any of the vertices of $\Gamma$ represented by one of the vertices in $\Omega(\Gamma)$ are structurally equivalent, so we just need to show that the vertices of $\Omega(\Gamma)$ are the complete equivalence classes. This follows from the fact that $\Omega(\Gamma)$ is the smallest possible reconfiguration of $\Gamma$ under the reinterpretation of vertices and edges. If there were structurally equivalent vertices in two distinct vertices $K_a$ and $K_b$ of $\Omega(\Gamma)$, we could conflate the two vertices into a single vertex $K_{a+b}$, where the set of edges incident to $K_{a+b}$ would be the set of edges incident to $K_a$ or $K_b$ (which are the same because they are structurally equivalent). However, because $\Omega(\Gamma)$ is the smallest possible reconfiguration of $\Gamma$ under the reinterpretation of vertices and edges, no vertices can be conflated in this way, so we deduce that the vertices are the complete equivalence classes.
\\
\\
Now, we will consider the case with disconnected structurally equivalent vertices separately. Clearly, if we have two or more $K_1$ vertices in $\Omega(\Gamma)$ that share the same neighbors, these vertices are not their own complete equivalence classes because we can form a larger one containing them. The paragraph above applies to the rest of the graph, though-- it is easy to observe that any of the remaining vertices of $\Omega(\Gamma)$ are still their own complete equivalence classes.
\end{proof}

\begin{corollaryN}
\label{no supersimilarity in complete skeleton}
No two connected vertices in $\Omega(\Gamma)$ are structurally equivalent.
\end{corollaryN}

\begin{proof}
This follows from reasoning similar to what was used to prove Theorem \ref{complete skeleton theorem}. If two vertices are connected and are structurally equivalent in $\Omega(\Gamma)$, then we can conflate them as we saw above, which contradicts the minimality of $V(\Omega(\Gamma))$.
\end{proof}

\bigskip
\noindent
\textbf{Definition.} We will let the \textbf{structure of a complete skeleton} of a graph $\Gamma$, which we will write as $S\Omega(\Gamma)$ or just $S\Omega$ when we are not concerned about the particular $\Gamma$, be the reinterpretation of $\Omega(\Gamma)$ where the vertices and edges are interpreted again in the standard way. 

\bigskip
To clarify, the structure of $\Omega(\Gamma)$ in Figure \ref{first complete skeleton ex.} is $C_4$. Note that $S\Omega(\Gamma)$ is what is typically called the compression graph of $\Gamma$. The vertices and edges of $S\Omega(\Gamma)$ are often called super-nodes and super-edges since they denote collections of nodes or edges in $\Gamma$. The primary value of complete skeletons, though, is that they uniquely represent graphs while maintaining the structures of their compressions.  

\begin{theoremN}
\label{complete skeleton suff nec}
For a graph $G$, $G = S\Omega$ if and only if $G$ has no connected structurally equivalent vertices.
\end{theoremN}

\begin{proof}
The forward direction was established by Corollary \ref{no supersimilarity in complete skeleton}. Now, we just need to show that any graph without connected structurally equivalent vertices can be the structure of a complete skeleton. Let $G$ be such a graph. Reinterpret the vertices and edges of $G$ such that the vertices represent complete induced subgraphs of some other graph $\Gamma$ and two of these vertices are connected if and only if the two complete induced subgraphs are completely connected. Because $G$ contains no connected structurally equivalent vertices, there are no connected vertices $K_a$ and $K_b$ in $R\Gamma$ such that $N^1(K_a)\setminus\{K_b\}\neq N^1(K_b)\setminus\{K_a\}$. Thus, no vertices in $R\Gamma$ can be conflated, so $R\Gamma = \Omega(\Gamma)$, which means $G=S\Omega$.
\end{proof}

\bigskip
These new results allow us to come up with a formula for rank$(I+A(\Gamma))$; next, we will find an upper bound for this.

\begin{theoremN}
\label{rank equation 2}
We can calculate rank$(I+A(\Gamma))$ by the equation:
\begin{equation}
\label{rank inequality}
    \textrm{rank}(I+A(\Gamma)) \leq |V(S\Omega(\Gamma))|.
\end{equation}
\end{theoremN}

\begin{proof}
Consider the matrix $I+A(\Gamma)$. This is precisely the adjacency matrix of $\Gamma$ where a loop is attached to each vertex. Let this new graph be $\hat{\Gamma}$. If two vertices $u$ and $v$ were connected and structurally equivalent in $\Gamma$, then we see $N^1(u)=N^1(v)$ in $\hat{\Gamma}$, so connected structurally equivalent vertices have identical columns in rows in $A(\hat{\Gamma})=A(I+A(\Gamma))$, so they must be linearly dependent. If $u$ and $v$ were disconnected structurally equivalent vertices in $\Gamma$, then $N^1(u) \cup \{v\}=N^1(v) \cup \{u\}$ in $\hat{\Gamma}$. This means that for disconnected structurally equivalent vertices, their columns and rows differ in two places (corresponding to the coordinates of the two vertices) in $A(\hat{\Gamma})$. Hence, the rows and columns for any disconnected structurally equivalent vertices must be linearly independent with respect to each other, so we cannot \textit{a priori} establish a lower upper bound for rank($I+A(\Gamma))$. Thus, we can conclude rank$(I+A(\Gamma))$ is bounded above by the sum of the number of connected equivalence classes in $\Gamma$ with the number of the individual disconnected structurally equivalent vertices. Equivalently, rank$(I+A(\Gamma)) \leq |V(S\Omega(\Gamma))|$.
\end{proof}

\bigskip
Using this result, we can also find a new expression for determining whether a graph has -1 as an eigenvalue and what its multiplicity is.


\begin{lemmaN}
\label{multiplicity inequality}
If the multiplicity of the -1 eigenvalue for a simple graph $\Gamma$ is $k$, then $|V(\Gamma)|-|V(S\Omega(\Gamma))|\geq k$.
\end{lemmaN}
\begin{proof}
It is well known in spectral graph theory that, for a graph on $n$ vertices, the multiplicity of an eigenvalue $\lambda$ is given by $n-\textrm{rank}(I\lambda-A(\Gamma))$, so the theorem follows from this and Theorem \ref{rank equation 2}. 
\end{proof}

Because the inequality in Theorem \ref{rank equation 2} provides us with a lower bound, we can correct this with a small constant term. If we denote this term by $\Lambda \in \mathbb{Z}^{\geq 0}$, then we find
\begin{equation}
\label{multiplicity equation}
rank(I+A(\Gamma))= |V(S\Omega(\Gamma))|-\Lambda = \textrm{rank}(I+A(S\Omega(\Gamma))),
\end{equation}
so
\begin{equation}
\Lambda=|V(S\Omega(\Gamma))|-\textrm{rank}(I+A(S\Omega(\Gamma))).
\end{equation}

Because we now have an equality, we can provide the following characterization of graphs by the multiplicity of the -1 eigenvalue.

\begin{theoremN}
\label{multiplicity characterization}
A graph $\Gamma$ has -1 as an eigenvalue of multiplicity $k$ if and only if 
\begin{equation}
\label{actual char equation}
k=|V(\Gamma)|-(|V(\Omega(\Gamma))|-\Lambda) = |V(\Gamma)|-\textrm{rank}(I+A(S\Omega(\Gamma))).
\end{equation}
\end{theoremN}

\begin{proof}
This follows by the same reasoning used in the proof of Theorem \ref{multiplicity inequality} and by Equation \ref{multiplicity equation}
\end{proof}

Unfortunately, this does not give us a purely graph theoretic method for determining the multiplicity of the -1 eigenvalue. For this to be the case, we would need a way to interpret $\Lambda$ without determining $\textrm{rank}(I+A(S\Omega(\Gamma)))$ for each complete skeleton structure. At least for small $|V(S\Omega(\Gamma))|$, we find that $\Lambda$ is almost always 0. As stated in the next section, the only complete skeleton structure with five or fewer vertices such that $\Lambda \neq 0$ is $P_5$, for which $\Lambda=1$. For all the other cases where $\Lambda=0$, though, we can rewrite Equation \ref{actual char equation} as 
\[
\textrm{rank}(I+A(\Gamma))=|V(\Gamma)|-|V(S\Omega(\Gamma))|,
\]
which effectively reduces the linear algebraic problem to a graph theoretic problem. For the general case, however, further research must be completed before we can solve this problem of determining the multiplicity of the -1 eigenvalue in purely graph theoretic terms. 

\section{Complete Skeletons and the multiplicity of the -1 eigenvalue}

Previous literature has already considered the possible graph structures such that $1 \leq \textrm{rank}(I+A(\Gamma)) \leq 3$. We will present these known results, translate them into the language of complete skeletons, and then we will then use some of the results from the previous section to consider some possible graphs such that rank$(I+A(\Gamma))=4$ or 5. Or, in other words, for graphs on $n$ vertices, we will find graphs that have -1 as an eigenvalue with multiplicity $n-4$ or $n-5$.

\begin{theoremN}
\label{rank 1}
If rank$(I+A(\Gamma))=1$, then $\Gamma$ is complete.
\end{theoremN}

\begin{theoremN}
\label{rank 2}
If rank$(I+A(\Gamma))=2$, then $\Gamma$ is the union of two disjoint complete graphs.
\end{theoremN}

\begin{theoremN}
\label{rank 3}
If rank$(I+A(\Gamma))=3$ such that $\Gamma$ has $n$ vertices, then $\Gamma = K_n\setminus K_{l,m}$ such that $l,m\geq 1$ and $l+m \leq n-1$ or $\Gamma = K_a+K_b+K_c$ such that $a,b,c \geq 1$ and $a+b+c=n$.
\end{theoremN}

These three theorems can be found in \cite{almost_complete_spectra}. The following three results are the equivalent expressions of the theorems above, using the terminology of complete skeletons. They will be presented without proof, for it should be easy to see how they follow from Theorems \ref{rank 1}, \ref{rank 2}, and \ref{rank 3}.

\bigskip
\noindent
\textbf{Theorem 5.1'.} \textit{If rank$(I+A(\Gamma))=1$, then $\Omega(\Gamma)$ is a single vertex.}

\bigskip
\noindent
\textbf{Theorem 5.2'.} \textit{If rank$(I+A(\Gamma))=2$, then $\Omega(\Gamma)$ is two disconnected vertices.}

\bigskip
\noindent
\textbf{Theorem 5.3'.} \textit{If rank$(I+A(\Gamma))=3$, then $\Omega(\Gamma)$ is three disconnected vertices or a 3-path.}

\bigskip
Next, we will consider graphs $\Gamma$ such that rank$(I+A(\Gamma))=4$. Because our characterization in Theorem \ref{multiplicity characterization} is not purely graph theoretic, we are unable to make the same kind of statement for these higher rank cases. In particular, our conditionals must consist of graph theoretic properties in the antecedent with the linear algebraic term as the consequent. Also, for the remainder of the paper, our inquiry will be guided by the number of vertices in complete skeletons.

\begin{theoremN}
\label{rank 4 theorem}
If $\Omega(\Gamma)$ is one of the graphs in Figure \ref{rank 4 fig}, then rank$(I+A(\Gamma))=4$.
\end{theoremN}

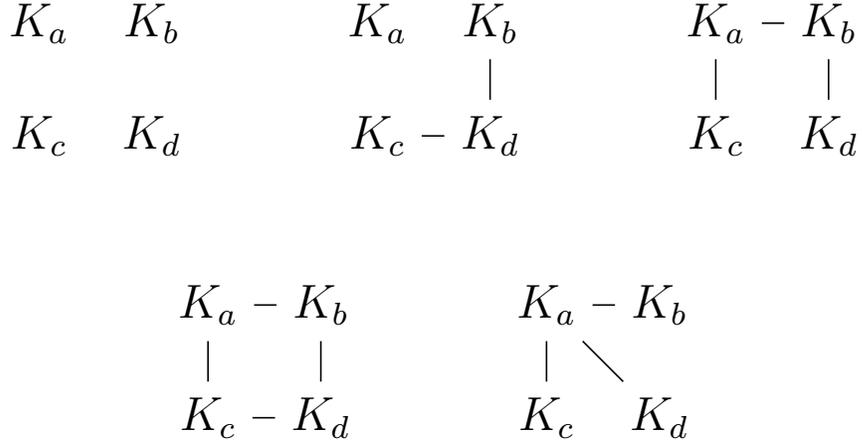
\begin{figure}
\scalebox{1.5}{
    \centering
    \begin{tikzpicture}
    \node (n1) at (2.5,1) {$K_c$};
    \node (n2) at (2.5,2) {$K_a$};
    \node (n3) at (3.5,1) {$K_d$};
    \node (n4) at (3.5,2) {$K_b$};
    \node (n5) at (5.5,1) {$K_c$};
    \node (n6) at (5.5,2) {$K_a$};
    \node (n7) at (6.5,1) {$K_d$};
    \node (n8) at (6.5,2) {$K_b$};
    \node (n13) at (1,3.5) {$K_c$};
    \node (n14) at (1,4.5) {$K_a$};
    \node (n15) at (2,3.5) {$K_d$};
    \node (n16) at (2,4.5) {$K_b$};
    \node (n17) at (4,3.5) {$K_c$};
    \node (n18) at (4,4.5) {$K_a$};
    \node (n19) at (5,3.5) {$K_d$};
    \node (n20) at (5,4.5) {$K_b$};
    \node (n21) at (7,3.5) {$K_c$};
    \node (n22) at (7,4.5) {$K_a$};
    \node (n23) at (8,3.5) {$K_d$};
    \node (n24) at (8,4.5) {$K_b$};
    
    \foreach \from/\to in {n17/n19,n19/n20,n21/n22,n22/n24,n24/n23,n1/n2,n2/n4,n3/n4,n1/n3,n6/n5,n6/n7,n6/n8}
    \draw (\from) -- (\to);
    \end{tikzpicture}}
    \caption{All the possible complete skeletons with exactly four vertices. The structures of these complete skeletons are $4K_1, K_1+P_3, P_4, C_4, K_{1,3},$ and $K_{2,2}$.}
    \label{rank 4 fig}
\end{figure}

\begin{proof}
The graphs in Figure \ref{rank 4 fig} are all the possible complete skeletons of four vertices. This can be verified by checking all graphs of four vertices and finding the ones that have no connected structurally equivalent vertices. By Theorem \ref{complete skeleton suff nec}, these graphs can be complete skeleton structures of graphs. Additionally, by Theorem \ref{rank equation 2}, we know that for any graph $\Gamma$ that has one of these five graphs as its complete skeleton, rank$(I+A(\Gamma)) \leq 4$. It can easily be shown that $\Lambda=|V(S\Omega(\Gamma))|-\textrm{rank}(I+A(S\Omega(\Gamma)))=0$ for all of these graphs, so rank$(I+A(\Gamma))=|V(S\Omega(\Gamma))|=4$.  
\end{proof}

Of course, it follows from this result that, if a graph $\Gamma$ of $n$ vertices has one of the graphs in Figure \ref{rank 4 fig} as its complete skeleton, then -1 is an eigenvalue of $\Gamma$ with multiplicity $n-4$. Figures \ref{r4 disconnected} and \ref{r4 connected} provide several examples of graphs that fall into this category.
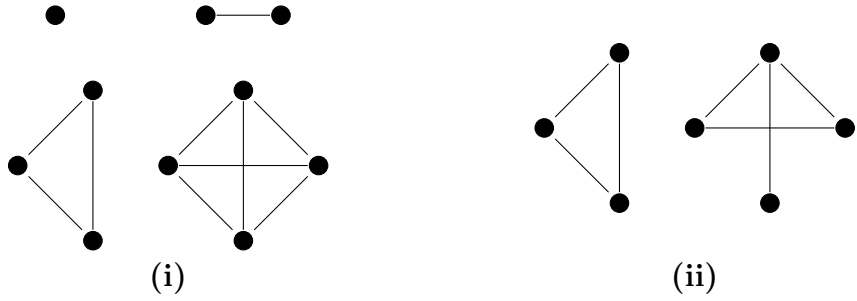
\begin{figure}
    \centering
    \begin{tikzpicture}
    \node (n1) at (1.5,4) {};
    \node (n2) at (3.5,4) {};
    \node (n3) at (4.5,4) {};
    \node (n4) at (2,3) {};
    \node (n5) at (1,2) {};
    \node (n6) at (2,1) {};
    \node (n7) at (3,2) {};
    \node (n8) at (4,3) {};
    \node (n9) at (4,1) {};
    \node (n10) at (5,2) {};
    \node (n11) at (8,2.5) {};
    \node (n12) at (9,1.5) {};
    \node (n13) at (9,3.5) {};
    \node (n14) at (10,2.5) {};
    \node (n15) at (11,1.5) {};
    \node (n16) at (11,3.5) {};
    \node (n17) at (12, 2.5) {};

     \draw [fill=black] (1.5,4) circle (3.5pt);
     \draw [fill=black] (3.5,4) circle (3.5pt);
     \draw [fill=black] (4.5,4) circle (3.5pt);
     \draw [fill=black] (2,3) circle (3.5pt);
     \draw [fill=black] (1,2) circle (3.5pt);
      \draw [fill=black] (2,1) circle (3.5pt);
       \draw [fill=black] (3,2) circle (3.5pt);
        \draw [fill=black] (4,3) circle (3.5pt);
         \draw [fill=black] (4,1) circle (3.5pt);
          \draw [fill=black] (5,2) circle (3.5pt);
          \draw [fill=black] (8,2.5) circle (3.5pt);
         \draw [fill=black] (9,1.5) circle (3.5pt);
          \draw [fill=black] (9,3.5) circle (3.5pt);
          \draw [fill=black] (10,2.5) circle (3.5pt);
        \draw [fill=black] (11,3.5) circle (3.5pt);
        \draw [fill=black] (11,1.5) circle (3.5pt);
        \draw [fill=black] (12,2.5) circle (3.5pt);
        
\node (n18) at (3,0.5) {\textbf{(i)}};
\node (n19) at (10,0.5) {\textbf{(ii)}};

    \foreach \from/\to in {n2/n3,n4/n5,n4/n6,n5/n6,n7/n8,n7/n9,n7/n10,n8/n9,n8/n10,n9/n10,n11/n12,n11/n13,n12/n13,n14/n16,n14/n17,n15/n16,n16/n17}
    \draw (\from) -- (\to);
    \end{tikzpicture}
    \caption{Two examples of disconnected graphs such that rank$(I+A(\Gamma))=4$. The graph in (i) is $K_1+K_2+K_3+K_4$ and has -1 as an eigenvalue of multiplicity 6. The graph in (ii) is $K_3 + K_4 \setminus K_{2,1}$ (or, identically, $K_3 + K_4 \setminus P_2$) and has -1 as an eigenvalue with multiplicity 3. Also, note $S\Omega(K_1+K_2+K_3+K_4) = 4K_1$ and $S\Omega(K_3 + K_4 \setminus K_{2,1}) = K_1+P_3$.}
    \label{r4 disconnected}
\end{figure}

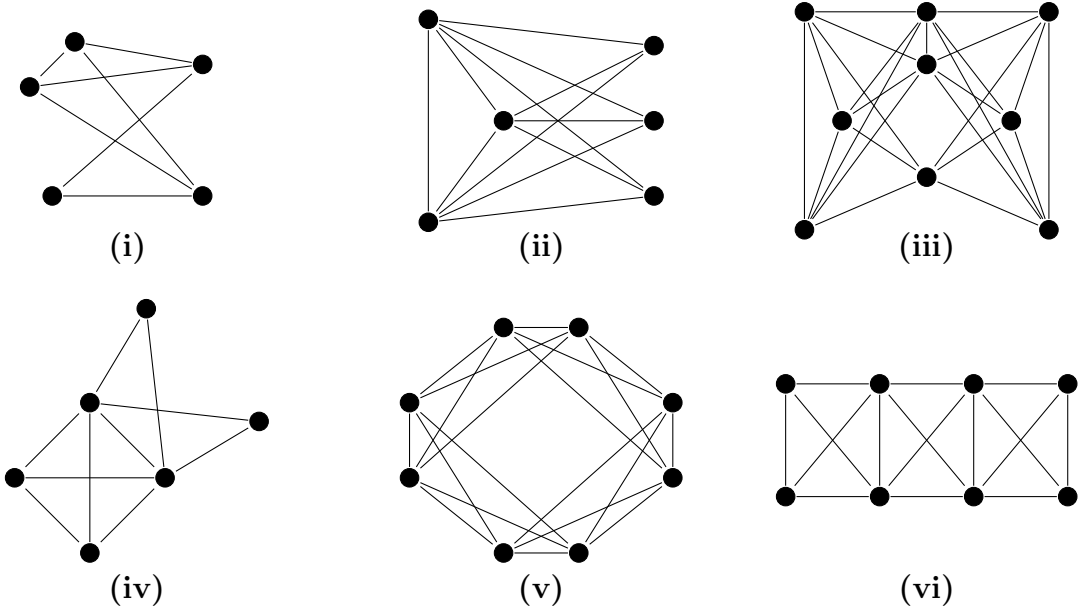
\begin{figure}
    \centering
    \begin{tikzpicture}
     \node (n1) at (0.5,2) {};
    \node (n2) at (1.5,1) {};
    \node (n3) at (1.5,3) {};
    \node (n4) at (2.5,2) {};
    \node (n5) at (2.25,4.25) {};
    \node (n6) at (3.75,2.75) {};
    \node (n7) at (5.75,2) {};
    \node (n8) at (5.75,3) {};
    \node (n9) at (7,1) {};
    \node (n10) at (8,1) {};
    \node (n11) at (9.25,2) {};
    \node (n12) at (9.25,3) {};
    \node (n13) at (7,4) {};
    \node (n14) at (8,4) {};
    \node (n15) at (10.75,1.75) {};
    \node (n16) at (10.75,3.25) {};
    \node (n17) at (12,1.75) {};
    \node (n18) at (12,3.25) {};
    \node (n19) at (13.25,1.75) {};
    \node (n20) at (13.25,3.25) {};
    \node (n21) at (14.5,1.75) {};
    \node (n22) at (14.5,3.25) {};
    \node (n23) at (1,5.75) {};
    \node (n24) at (0.7,7.2) {};
    \node (n47) at (1.3,7.8) {};
    \node (n25) at (3,5.75) {};
    \node (n26) at (3,7.5) {};
     \node (n27) at (6,5.4) {};
    \node (n28) at (6,8.1) {};
    \node (n29) at (7,6.75) {};
    \node (n30) at (9,5.75) {};
    \node (n31) at (9,6.75) {};
    \node (n32) at (9,7.75) {};
     \node (n33) at (11,8.2) {};
    \node (n34) at (11,5.3) {};
     \node (n35) at (11.5,6.75) {};
    \node (n36) at (12.625,8.2) {};
    \node (n48) at (12.625,7.5) {};
    \node (n37) at (12.625,6) {};
    \node (n38) at (13.75,6.75) {};
    \node (n39) at (14.25,8.2) {};
    \node (n40) at (14.25,5.3) {};
    
    \draw [fill=black] (0.5,2) circle (3.5pt);
     \draw [fill=black] (1.5,1) circle (3.5pt);
     \draw [fill=black] (1.5,3) circle (3.5pt);
     \draw [fill=black] (2.5,2) circle (3.5pt);
     \draw [fill=black] (2.25,4.25) circle (3.5pt);
      \draw [fill=black] (3.75,2.75) circle (3.5pt);
        \draw [fill=black] (5.75,2) circle (3.5pt);
     \draw [fill=black] (5.75,3) circle (3.5pt);
     \draw [fill=black] (7,1) circle (3.5pt);
     \draw [fill=black] (8,1) circle (3.5pt);
     \draw [fill=black] (9.25,2) circle (3.5pt);
      \draw [fill=black] (9.25,3) circle (3.5pt);
      \draw [fill=black] (7,4) circle (3.5pt);
      \draw [fill=black] (8,4) circle (3.5pt);
        \draw [fill=black] (10.75,1.75) circle (3.5pt);
     \draw [fill=black] (10.75,3.25) circle (3.5pt);
     \draw [fill=black] (12,1.75) circle (3.5pt);
     \draw [fill=black] (12,3.25) circle (3.5pt);
     \draw [fill=black] (13.25,1.75) circle (3.5pt);
      \draw [fill=black] (13.25,3.25) circle (3.5pt);
      \draw [fill=black] (14.5,1.75) circle (3.5pt);
      \draw [fill=black] (14.5,3.25) circle (3.5pt);
       \draw [fill=black] (1,5.75) circle (3.5pt);
      \draw [fill=black] (0.7,7.2) circle (3.5pt);
      \draw [fill=black] (1.3,7.8) circle (3.5pt);
      \draw [fill=black] (3,5.75) circle (3.5pt);
      \draw [fill=black] (3,7.5) circle (3.5pt);
      \draw [fill=black] (6,5.4) circle (3.5pt);
      \draw [fill=black] (6,8.1) circle (3.5pt);
       \draw [fill=black] (7,6.75) circle (3.5pt);
      \draw [fill=black] (9,5.75) circle (3.5pt);
      \draw [fill=black] (9,6.75) circle (3.5pt);
      \draw [fill=black] (9,7.75) circle (3.5pt);
       \draw [fill=black] (11,8.2) circle (3.5pt);
      \draw [fill=black] (11,5.3) circle (3.5pt);
      \draw [fill=black] (11.5,6.75) circle (3.5pt);
      \draw [fill=black] (12.625,8.2) circle (3.5pt);
      \draw [fill=black] (12.625,7.5) circle (3.5pt);
       \draw [fill=black] (12.625,6) circle (3.5pt);
      \draw [fill=black] (13.75,6.75) circle (3.5pt);
      \draw [fill=black] (14.25,8.2) circle (3.5pt);
      \draw [fill=black] (14.25,5.3) circle (3.5pt);
      
      \node (n41) at (2.125,0.5) {\textbf{(iv)}};
      \node (n42) at (7.5,0.5) {\textbf{(v)}};
      \node (n43) at (12.625,0.5) {\textbf{(vi)}};
      \node (n44) at (2,5.05) {\textbf{(i)}};
      \node (n45) at (7.5,5.05) {\textbf{(ii)}};
      \node (n46) at (12.625,5.05) {\textbf{(iii)}};
    
    \foreach \from/\to in {n1/n2,n1/n3,n1/n4,n2/n3,n2/n4,n3/n4,n5/n3,n5/n4,n6/n3,n6/n4,n7/n8,n9/n10,n11/n12,n13/n14,n7/n9,n7/n10,n8/n9,n8/n10,n9/n11,n9/n12,n10/n11,n10/n12,n11/n13,n11/n14,n12/n13,n12/n14,n13/n7,n13/n8,n14/n7,n14/n8,n15/n16,n17/n18,n19/n20,n21/n22,n15/n17,n15/n18,n16/n17,n16/n18,n17/n19,n17/n20,n18/n19,n18/n20,n19/n21,n19/n22,n20/n21,n20/n22,n23/n25,n23/n26,n24/n25,n24/n26,n47/n25,n47/n26,n47/n24,n27/n28,n27/n29,n28/n29,n30/n27,n30/n28,n30/n29,n31/n27,n31/n28,n31/n29,n32/n27,n32/n28,n32/n29,n33/n34,n33/n35,n34/n35,n38/n39,n38/n40,n39/n40,n36/n33,n36/n34,n36/n35,n36/n38,n36/n39,n36/n40,n37/n33,n37/n34,n37/n35,n37/n38,n37/n39,n37/n40,n48/n36,n48/n33,n48/n34,n48/n35,n48/n38,n48/n39,n48/n40}
    \draw (\from) -- (\to);
    \end{tikzpicture}
    \caption{Some examples of connected graphs included in Theorem \ref{rank 4 theorem}. Let $\Gamma_i$ denote the $i$th graph. Then, $S\Omega(\Gamma_1) = K_{2,2}, S\Omega(\Gamma_2) = K_{1,3}, S\Omega(\Gamma_3) = C_4, S\Omega(\Gamma_4) = K_{1,3}, S\Omega(\Gamma_5) = C_4,$ and $S\Omega(\Gamma_6) = P_4$.}
    \label{r4 connected}
\end{figure}

\bigskip
We can follow a very similar procedure for determining graphs such that rank$(I+A(\Gamma))=5$ through finding their possible complete skeletons. This is what we will accomplish in the next theorem.

\begin{theoremN}
Figure \ref{complete skeletons rank 5} contains all complete skeletons of five vertices. Further, if $\Omega(\Gamma)$ is one of the graphs in Figure \ref{complete skeletons rank 5}, then rank$(I+A(\Gamma))=5$ with one exception. If $S\Omega(\Gamma)=P_5$, then rank$(I+A(\Gamma))=4$.
\end{theoremN}

\begin{proof}
There are 34 total graphs with five vertices, which can all be found in \cite{graphclasses}. As we saw in the proof of Theorem \ref{rank 4 theorem}, we can find the possible complete skeletons of five vertices by determining all graphs with five vertices that contain no two connected structurally equivalent vertices; these graphs are the structures of the complete skeletons. The reader is encouraged to verify that the fifteen graphs in Figure \ref{complete skeletons rank 5} are all such graphs. Further, it can easily be verified with a CAS that $\Lambda=|V(S\Omega(\Gamma))|-\textrm{rank}(I+A(S\Omega(\Gamma)))=0$ for all of the graphs in Figure \ref{complete skeletons rank 5}, excluding the second one in the second row. In other words, rank$(I+A(\Gamma))=5$ unless $S\Omega(\Gamma)=P_5$. In this case, there is linear dependence among the columns of $I+A(S\Omega(\Gamma))$ and, in particular, rank$(I+A(S\Omega(\Gamma)))=4$, so $\Lambda=1$ and rank$(I+A(\Gamma))=4$ by Equation \ref{multiplicity equation}.
\end{proof}
Clearly, it grows increasingly tedious to find all of the possible complete skeletons of graphs with a given number of vertices as this number grows larger. Additionally, it is not particularly interesting, so we will not proceed to consider the six vertex case, even though this could be done. If one were to pursue this question, the theory has now been sufficiently developed to do so with relative ease. As we have seen, all one needs to do in order to find all of the possible complete skeleton structures with $k$ vertices is to find all of the graphs of $k$ vertices such that no two connected vertices are structurally equivalent and let these be the structures of the possible complete skeletons. Also, as we have seen, it is one thing to determine all the possible complete skeletons for a given number of vertices, but it would require further work to find $\Lambda$ for all these graphs. By what we saw, it seems reasonable to conjecture that $\Lambda$ is "often" 0, but it is difficult at this point to tell how strong "often" is. It certainly is too early to conjecture that $\Lambda$ is almost always 0. However, if the problem of determining $\Lambda$ in general could be solved, this would provide a characterization of graphs by the multiplicity of the -1 eigenvalue that is clean and easy to work with.

\section{Acknowledgements}

This work originated while the author participated in an REU-program (NSF-REU grant DMS-1757233) run virtually at Texas State University 
during Summer 2020, directed by Yong Yang (PI) and Thomas M. Keller (co-PI).

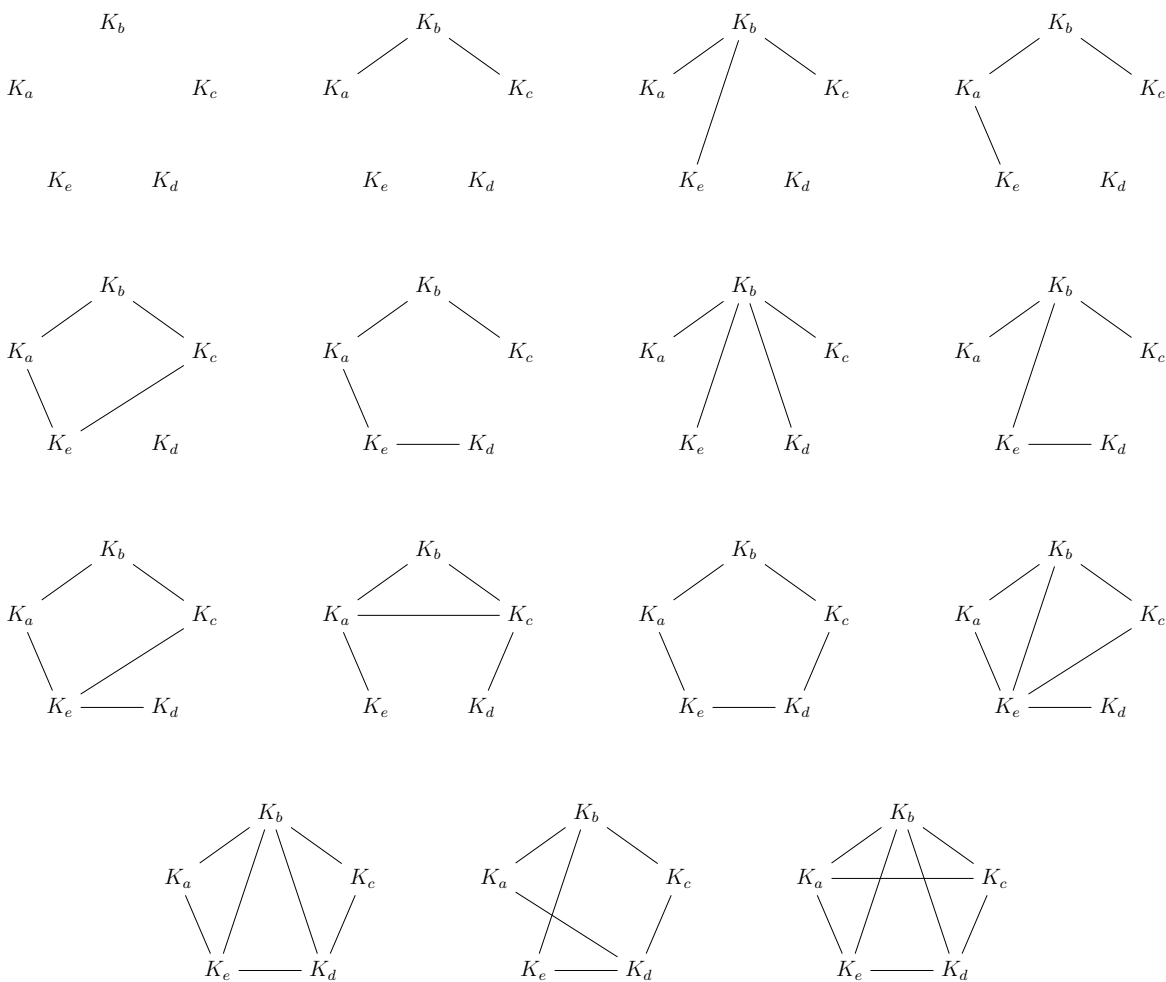
\begin{figure}

    \centering
    \scalebox{0.7}{
    \begin{tikzpicture}
     \node (n1) at (0.75,2.75) {$K_a$};
    \node (n2) at (1.5,1) {$K_e$};
    \node (n3) at (3.5,1) {$K_d$};
    \node (n4) at (4.25,2.75) {$K_c$};
    \node (n5) at (2.5,4) {$K_b$};
        \node (n6) at (6.75,2.75) {$K_a$};
    \node (n7) at (7.5,1) {$K_e$};
    \node (n8) at (9.5,1) {$K_d$};
    \node (n9) at (10.25,2.75) {$K_c$};
    \node (n10) at (8.5,4) {$K_b$};
     \node (n11) at (12.75,2.75) {$K_a$};
    \node (n12) at (13.5,1) {$K_e$};
    \node (n13) at (15.5,1) {$K_d$};
    \node (n14) at (16.25,2.75) {$K_c$};
    \node (n15) at (14.5,4) {$K_b$};
     \node (n16) at (18.75,2.75) {$K_a$};
    \node (n17) at (19.5,1) {$K_e$};
    \node (n18) at (21.5,1) {$K_d$};
    \node (n19) at (22.25,2.75) {$K_c$};
    \node (n20) at (20.5,4) {$K_b$};
    
    \node (n21) at (0.75,7.75) {$K_a$};
    \node (n22) at (1.5,6) {$K_e$};
    \node (n23) at (3.5,6) {$K_d$};
    \node (n24) at (4.25,7.75) {$K_c$};
    \node (n25) at (2.5,9) {$K_b$};
        \node (n26) at (6.75,7.75) {$K_a$};
    \node (n27) at (7.5,6) {$K_e$};
    \node (n28) at (9.5,6) {$K_d$};
    \node (n29) at (10.25,7.75) {$K_c$};
    \node (n30) at (8.5,9) {$K_b$};
     \node (n31) at (12.75,7.75) {$K_a$};
    \node (n32) at (13.5,6) {$K_e$};
    \node (n33) at (15.5,6) {$K_d$};
    \node (n34) at (16.25,7.75) {$K_c$};
    \node (n35) at (14.5,9) {$K_b$};
     \node (n36) at (18.75,7.75) {$K_a$};
    \node (n37) at (19.5,6) {$K_e$};
    \node (n38) at (21.5,6) {$K_d$};
    \node (n39) at (22.25,7.75) {$K_c$};
    \node (n40) at (20.5,9) {$K_b$};
    
    \node (n41) at (0.75,12.75) {$K_a$};
    \node (n42) at (1.5,11) {$K_e$};
    \node (n43) at (3.5,11) {$K_d$};
    \node (n44) at (4.25,12.75) {$K_c$};
    \node (n45) at (2.5,14) {$K_b$};
        \node (n46) at (6.75,12.75) {$K_a$};
    \node (n47) at (7.5,11) {$K_e$};
    \node (n48) at (9.5,11) {$K_d$};
    \node (n49) at (10.25,12.75) {$K_c$};
    \node (n50) at (8.5,14) {$K_b$};
     \node (n51) at (12.75,12.75) {$K_a$};
    \node (n52) at (13.5,11) {$K_e$};
    \node (n53) at (15.5,11) {$K_d$};
    \node (n54) at (16.25,12.75) {$K_c$};
    \node (n55) at (14.5,14) {$K_b$};
     \node (n56) at (18.75,12.75) {$K_a$};
    \node (n57) at (19.5,11) {$K_e$};
    \node (n58) at (21.5,11) {$K_d$};
    \node (n59) at (22.25,12.75) {$K_c$};
    \node (n60) at (20.5,14) {$K_b$};
    
    \node (n61) at (3.75,-2.25) {$K_a$};
    \node (n62) at (4.5,-4) {$K_e$};
    \node (n63) at (6.5,-4) {$K_d$};
    \node (n64) at (7.25,-2.25) {$K_c$};
    \node (n65) at (5.5,-1) {$K_b$};
     \node (n66) at (9.75,-2.25) {$K_a$};
    \node (n67) at (10.5,-4) {$K_e$};
    \node (n68) at (12.5,-4) {$K_d$};
    \node (n69) at (13.25,-2.25) {$K_c$};
    \node (n70) at (11.5,-1) {$K_b$};
     \node (n71) at (15.75,-2.25) {$K_a$};
    \node (n72) at (16.5,-4) {$K_e$};
    \node (n73) at (18.5,-4) {$K_d$};
    \node (n74) at (19.25,-2.25) {$K_c$};
    \node (n75) at (17.5,-1) {$K_b$};
    
    \foreach \from/\to in {n46/n50,n50/n49,n51/n55,n55/n54,n55/n52,n56/n60,n60/n59,n56/n57,n21/n25,n25/n24,n21/n22,n22/n24,n26/n30,n30/n29,n26/n27,n27/n28,n31/n35,n35/n32,n35/n33,n35/n34,n36/n40,n40/n39,n40/n37,n37/n38,n6/n10,n10/n9,n6/n9,n6/n7,n9/n8,n11/n15,n15/n14,n14/n13,n13/n12,n12/n11,n1/n5,n5/n4,n1/n2,n2/n4,n2/n3,n16/n20,n20/n19,n19/n17,n16/n17,n17/n18,n17/n20,n61/n65,n65/n64,n64/n63,n63/n62,n65/n62,n65/n63,n62/n61,n66/n70,n70/n69,n69/n68,n68/n67,n67/n70,n66/n68,n71/n75,n75/n74,n74/n73,n73/n72,n72/n71,n71/n74,n75/n72,n75/n73}
    \draw (\from) -- (\to);
    \end{tikzpicture}}
    \caption{All the possible complete skeletons of graphs such that rank$(I+A(\Gamma))=5$.}
    \label{complete skeletons rank 5}
\end{figure}

\printbibliography
\end{document}